\documentclass[12pt]{article}%
   
\usepackage{amsmath,enumerate}
\usepackage{amsfonts}
\usepackage{amssymb}

\newtheorem{theorem}{Theorem}[section]

\newtheorem{definition}[theorem]{Definition}

\newtheorem{lemma}[theorem]{Lemma}

\newenvironment{proof}[1][Proof]{\noindent\textbf{#1.} }
{\hfill \ \rule{0.5em}{0.5em}}

\newcommand{\modn}{(\textup{mod}~n)}

\begin{document}

\title{A note on the number of edges in a Hamiltonian graph with no repeated cycle length}

\author{Joey Lee\thanks{Department of Mathematics and Statistics, California State University Sacramento, \texttt{joeylee@csus.edu}}
\and 
Craig Timmons\thanks{Department of Mathematics and Statistics, California State University Sacramento, \texttt{craig.timmons@csus.edu}.
Research supported in part by Simons Foundation Grant \#359419.}}

\maketitle

\abstract{
Let $G$ be an $n$-vertex graph obtained by adding chords to a cycle of length $n$.  Markstr\"{o}m 
asked for the maximum number of edges in $G$ if there are no two cycles in $G$ with the same length.
A simple counting argument shows that such a graph can have at most $n  + \sqrt{2n} +1 $ edges.
Using difference sets in $\mathbb{Z}_n$, we show that for infinitely many $n$, 
there is an $n$-vertex Hamiltonian graph with $n + \sqrt{n - 3/4} - 3/2$ edges and no repeated cycle length.    
}


\section{Introduction}

Let $G$ be a graph with $n$ vertices.  The \emph{cycle spectrum} of $G$, which we denote by $\mathcal{S}(G)$, 
is the set of all $k \in \{3,4, \dots ,n \}$ for which there is a cycle of length $k$ in $G$.
Some of the most basic structural properties of a graph can be stated in terms of the cycle 
spectrum: $G$ is bipartite if and only if $\mathcal{S}(G)$ contains no odd integer, 
$G$ is a tree if and only if $\mathcal{S}(G) = \emptyset$, $G$ is Hamiltonian if and only if 
$n \in \mathcal{S}(G)$, and finally, $G$ is pancyclic if and only if 
$\mathcal{S}(G)= \{3,4, \dots , n \}$.  
In addition, one may also be interested in the number of cycles in $G$ of a given length.  In 
this case, it is natural to consider the multiset version of $\mathcal{S}(G)$, which we denote 
by $\mathcal{S}^m(G)$.  More precisely, an integer $k$ appears $l$ times in $\mathcal{S}^m (G)$ 
if $G$ has exactly $l$ cycles of length $k$.  The $n$-vertex graph $G$ is \emph{uniquely pancyclic} if 
\[
\mathcal{S}^m (G) = \{3,4, \dots , n \}.
\]
That is, for every $k \in \{3,4, \dots , n \}$, $G$ has exactly one cycle of length $k$.  
The existence of uniquely pancyclic graphs has been studied by Shi \cite{shi} and Markstr\"{o}m \cite{markstrom}.  
It is an open problem of whether or not there exists infinitely many uniquely pancyclic graphs.  In addition to 
existence questions, one can also ask extremal type questions such as how large or small
$\mathcal{S}(G)$ can be given that $G$ has a fixed number of edges (see \cite{bkm} Problem 4.3).  In the other direction, what is the 
maximum number of edges in an $n$-vertex graph with no repeated cycle length?  
According to Bondy and Murty, this question was asked by Erd\H{o}s (see \cite{bm}, Problem 11 on page 247).  
Let us write $f(n)$ for the maximum number of edges in an $n$-vertex graph with no repeated cycle length. 
The best known lower bound on $f(n)$ is due to Lai \cite{lai} who 
proved that 
\begin{equation}\label{lai result}
\liminf_{n \rightarrow \infty} \frac{f(n) - n }{ \sqrt{n} }  \geq 
\sqrt{2.4 }.
\end{equation} 
A result of Boros, Caro, F\"{u}redi, and Yuster \cite{boros} implies that 
\[
f(n) \leq n + (1.98 + o(1)) \sqrt{n}.
\]
Lai has conjectured that $\lim_{n \rightarrow \infty} \frac{ f(n) - n }{ \sqrt{n}  } = \sqrt{2.4}$ and determining
an asymptotic formula for $f(n)$ is an unsolved problem.         

While studying uniquely pancyclic graphs, Markstr\"{o}m \cite{markstrom} proved an upper bound on the number of chords 
in such a graph (since a uniquely pancyclic graph must contain a cycle of length $n$, it may be 
constructed by adding chords to a cycle of length $n$).  
His upper bound is exceeded by Lai's lower bound on $f(n)$.  The reason 
for this is that Lai's construction that proves (\ref{lai result}) has no cycle of length $\Omega (n)$.  This motivated 
Markstr\"{o}m to pose the problem of determining the maximum number of edges in an $n$-vertex Hamiltonian graph 
with no repeated cycle length (see Problem 2.2 in \cite{markstrom}). 
This problem was reiterated in the suvey of Lai and Liu \cite{lai-liu}.  
Let 
\[
g(n)
\]
be the maximum number edges in an $n$-vertex, Hamiltonian graph with no repeated cycle length.  
Using difference sets in the cyclic group $\mathbb{Z}_n$, we prove the following.  

\begin{theorem}\label{thm 1}
If $q$ is a power of a prime and $n = q^2 + q + 1$, then  
\[
g(n) \geq n + \sqrt{n - 3/4} - 3/2.
\]
\end{theorem}

If $G$ is an $n$-vertex graph with no repeated cycle length, then we can obtain an upper bound on the number of 
edges of $G$ as follows.  We view $G$ as obtained from $C_n$ by adding some number of chords, say $k$.  Each 
pair of chords determines at least one cycle and so, since no cycle length is repeated, $\binom{k}{2} < n$.
This gives the upper bound $g(n) < n + \sqrt{2n} + 1$.  This argument, which can be refined 
(see \cite{markstrom}) shows 
that Theorem \ref{thm 1} gives the correct order of magnitude of $g(n)$.  Like in the case with $f(n)$, we suspect that 
determining an asymptotic formula for $g(n)$ could be difficult, but at the same time, it is possible 
that adding the Hamiltonicity constraint makes the problem easier.  This is discussed further in the concluding remarks section.  In 
the next section, we give the proof of Theorem \ref{thm 1}.  


\section{Proof of Theorem \ref{thm 1}}

For $n \geq 3$, we write $C_n$ for the cycle of length $n$.  We will always assume that the vertices of $C_n$ are 
$\{1,2, \dots , n \}$, and the edges are $\{ i , i + 1 \}$ for $1 \leq i \leq n - 1$ together with $\{n , 1 \}$.  
If $C$ is a cycle whose edges are 
\[
\{x_1 , x_2 \}, \{x_2 , x_3 \} , \dots , \{x_{k-1} , x_k \} , \{x_k , x_1 \},
\]
then we write $x_1 , x_2 , x_3 , \dots , x_k , x_1$ for $C$. 

\begin{definition}
Given a positive integer $n \geq 4$ and a set $S \subseteq \{3,4, \dots , n -1 \}$, let $G_n (S)$ be the 
graph obtained by adding the edges $ \{1 , a \}$ to the cycle $C_n$ for each $a \in S$.  
\end{definition}

The graph that we construct will be obtained by choosing $S$ and $n$ carefully.  The next 
definition lists the conditions that we need $S$ to satisfy in order for $G_n (S)$ to have no repeated cycle length.  

\begin{definition}\label{def}
Let $n \geq 4$ and $S \subseteq \{3,4, \dots , n -1 \}$.  We say that $S$ is a \emph{distinct cycle set} 
if the following two conditions hold:
\begin{enumerate}
\item the differences $b-a$ with $b,a \in S$ and $b > a$ are all distinct,
\item the sets $S$, $S_n^{ \star }:= \{ n + 2 - a : a \in S \}$, and $S^{ - }  := \{ b - a + 2 : a,b \in S , b > a \}$,  
 are pairwise disjoint. 
\end{enumerate}
\end{definition}

\begin{lemma}\label{lemma 2}
If $n \geq 4$ and $S \subseteq \{3,4, \dots , n - 1 \}$ is a distinct cycle set, then the graph 
$G_n (S)$ has $n + |S|$ edges and no two cycles in $G_n (S)$ have the same length.
\end{lemma}
\begin{proof}
It is clear from the definitions that $G_n (S)$ has $n + |S|$ edges.
We must show that no two cycles have the same length.
Observe that all chords in $G_n(S)$ are incident to the vertex 1.  
Thus, any cycle in $G_n(S)$ must contain the vertex 1, and furthermore, 
must pass through exactly zero, one, or two chords.  There is only one cycle that 
contains no chords, namely $1,2,3, \dots , n , 1$.  There are two types of cycles that 
pass through exactly one chord.  Given $a \in S$, the sequence $1,2, \dots , a-1 , a, 1$ forms 
a cycle of length $a$ and we call this a cycle of Type 1.  
Also, the sequence $1, a, a+1 , \dots , n -1 , n , 1$ forms a cycle 
of length $n + 2 - a$ and we call a cycle of this form Type 2.  
The cycles in $G_n(S)$ that pass through exactly two chords, which we call Type 3, are of 
the form $1, a, a+1 , \dots , b -1 , b , 1$ where $a,b$ are elements of $S$ with $b > a$.       
In short, 
\begin{itemize} 
\item Type 1 are cycles of the form $1,2,3 \dots a,1$ where $a \in S$ and have length $a$,
\item Type 2 are cycles of the form $1, a, a+1, \dots ,n-1,n,1$ where $a \in S$ and have length $n  + 2  -a$, and
\item Type 3 are cycles of the form $1,a,a+1, \dots ,b-1,b,1$ where $b > a$ are in $S$ and have length $b - a + 2$.       
\end{itemize}
No two distinct cycles of Type 1 will have the same length.  If a Type 1 has the same length as a Type 2, then 
there are elements $a,b \in S $ with $a = n + 2 - b$.  This implies $ S \cap S_n^{ \star} \neq \emptyset$ which cannot occur 
since $S$ is a distinct cycle set.  Similarly, if a Type 1 has the same length as a Type 3, then 
$S \cap S^{ -} \neq \emptyset$ which cannot occur.  
No two distinct cycles of Type 2 will have the same length.  If a Type 2 has the same length as a Type 3, then 
$S_n^{ \star} \cap S^- \neq \emptyset$.  
Lastly, if two distinct cycles of Type 3 have the same length, then there are elements $a,b,c,d \in S$ with 
$b - a + 2 = d - c + 2$ and $b > a$, $d > c$.   This cannot occur since the 
differences $b-a$ with $b > a$ and $a,b \in S$ are all distinct.  
\end{proof}

\begin{lemma}\label{main lemma}
If $n \geq 4$ and $A \subset \mathbb{Z}_n$ is a perfect difference set, then there is a 
a distinct cycle set $S \subseteq \{3,4, \dots , n -1 \}$ with $|S| \geq |A| -2$.  
\end{lemma}
\begin{proof}
Let $A \subseteq \mathbb{Z}_n$ be a perfect difference set.  There is a unique ordered pair 
$(a_0   , b_0 ) \in A \times A$ such that $a_0 - b_0 \equiv 2 \modn$.  
Let 
\[
B = \{ a - b_0 \modn : a \in A \}.
\]
Since $B$ is a translate of $A$, $B$ is a perfect difference set in 
$\mathbb{Z}_n$.  
Observe that $B$ contains 
\begin{center}
$a_0 - b_0 \equiv 2 \modn$ and $b_0 - b_0 \equiv n \modn $.
\end{center} 
We may view $B$ as a subset of $\{1,2, \dots , n \}$ and we let $S = B \backslash \{ 2 , n \}$.
The set $S$ has at least $|A| - 2$ elements, and has the property that 
all of the differences $b -a$ with $b,a \in S$ and $b > a$ are distinct. 
To complete the proof, we must show that the sets $S$, $S^-$, and $S_n^{ \star }$ are pairwise disjoint (see Definition \ref{def} 
for the definitions of $S^-$ and $S_n^{ \star}$).
In each of the three cases, we will argue by contradiction.

If $S \cap S^- \neq \emptyset$, then there are elements 
$a,b,c \in S$ with $c = b - a + 2$ where $b > a$.  The equation $c = b - a + 2$ implies 
\[
c - 2 \equiv b - a \modn.
\]
Each of $a,b,c$, and 2 belong to $B$ and since $B$ is a difference set, we must have $c = 2$ or $b = 2$.
However, both $c$ and $b$ belong to $S$ and $2 \notin S$.  We conclude that $S \cap S^- = \emptyset$.

Suppose $S \cap S_n^{ \star} \neq \emptyset$.  There are elements $a,b \in S$ with 
$b = n + 2 - a$.  This implies $b \equiv n + 2 - a \modn$ so that $b -2 \equiv n - a \modn$.  
Now $b,2, n$, and $a$ are all elements of $B$ so that $b =2$ or $a=2$.  Again, this is a contradiction 
since $a,b \in S$ but $2 \notin S$.

Lastly, suppose that $S^- \cap S_n^{ \star} \neq \emptyset$.  There are elements 
$a,b,c \in S$ with 
$b - a + 2  = n + 2 - c$.  We can cancel 2 and then take the resulting equation modulo $n$ to get 
\[
b - a \equiv n - c \modn .
\]
The elements $b,a,n$, and $c$ all belong to $B$.  We must have $b = n$ or $n = c$ but $n \notin S$ so this 
cannot occur.

The preceding three paragraphs show that the sets $S$, $S^-$, and $S_n^{ \star}$ are all pairwise disjoint.
Therefore, $S$ is a distinct cycle set.  
\end{proof}

\bigskip

\begin{proof}[Proof of Theorem \ref{thm 1}]
Whenever $q$ is a power of a prime, there is a perfect difference set 
$A_q \subseteq \mathbb{Z}_{q^2 + q + 1}$ with $q + 1$ elements.  
This classical result is due to Singer \cite{singer}.  
By Lemma \ref{main lemma}, there is a distinct cycle set $S_q \subseteq \{3,4, \dots , q^2  + q \}$ with 
$|S_q| = q - 1$.  By Lemma \ref{lemma 2}, the graph 
$G_{q^2 + q + 1}(S_q)$ has $q^2 + 2q$ edges and no repeated cycle length.  Therefore, 
\[
g(q^2 + q + 1 ) \geq q^2 + 2q
\]
whenever $q$ is a power of a prime.  
\end{proof}


\section{Concluding Remarks}

Theorem \ref{thm 1} implies that 
\begin{equation}\label{eq 1.1}
\limsup_{n \rightarrow \infty } \frac{ g(n) - n }{ \sqrt{n } } \geq 1.
\end{equation}
Our approach, that of taking $C_n$ and adding chords incident to a single vertex, will not lead to improvements upon (\ref{eq 1.1}).  
This is because for any $S \subseteq \{3,4, \dots , n \}$, if $G_n (S)$ has no repeated cycle lengths, then 
$S$ forms a Sidon set in $\{1,2 \dots ,n \}$.  This is a set with the property that all sums of pairs of elements are distinct.  
A famous result of Erd\H{o}s and Tur\'{a}n \cite{et} says that a Sidon set in $\{1,2, \dots , n \}$ has at most $(1 + o (1)) \sqrt{n}$ elements 
and thus, $|S| \leq (1 + o(1)) \sqrt{n}$.

We also have $g(n) \leq n + \sqrt{2 n } + 1$.  
This upper bound comes from counting pairs of chords (see the introduction for 
more details) that have been added to $C_n$ to obtain $G$.  
We call of pair of chords $e$ and $e'$ \emph{crossing} if $C_n$ together with $e$ and $e'$ is a subdivision 
of $K_4$.  Otherwise, call $e$ and $e'$ \emph{non-crossing}.  A pair of non-crossing chords can be used to create 
just one cycle.  In a graph with no repeated cycle lengths, the crossing chords create two cycles 
with different lengths.  
If $k$ is the number of chords and $c$ is the number of pairs of chords that are crossing, then
\[
2 c + \left( \binom{k}{2} - c \right) \leq n .
\]
If $c >  \delta \binom{k}{2}$ for some $\delta > 0$, then we obtain an 
improvement on the bound 
\[
g(n) \leq n + \sqrt{2n } + 1
\]
which 
follows from the inequality $\binom{k}{2} \leq n$.  A construction improving (\ref{eq 1.1}) can therefore, not have 
too many crossing chords.


\end{document}